\documentclass[12pt]{amsart}

\usepackage[margin=1in]{geometry}
\usepackage{graphicx}
\usepackage{amsthm, amsmath, amssymb, tikz, bbm,amsfonts}
\usepackage{mathtools}
\usepackage{xifthen}

\usepackage[shortlabels]{enumitem}
\usepackage{todonotes}
\usetikzlibrary{calc,shapes, backgrounds}

\newcommand{\op}{\operatorname}
\newcommand{\floor}[1]{\left\lfloor #1 \right\rfloor}

\newcommand{\ra}{\rightarrow}

\newcommand{\class}{\mathcal{B}}
\def\Cr{{\mbox {cr}}}
\newcommand{\old}[1]{{}}
\newtheorem{defn}{Definition}

\newtheorem{quest}[defn]{Question}
\newtheorem{thm}[defn]{Theorem}
\newtheorem{lem}[defn]{Lemma}

\DeclarePairedDelimiter{\ceil}{\lceil}{\rceil}

\begin{document}

\title{Midrange crossing constants for graphs classes}

\author{\'Eva Czabarka}
\author{Josiah Reiswig}
\author{L\'aszl\'o Sz\'ekely}
\author{Zhiyu Wang}

\address{\'Eva Czabarka\\Department of Mathematics \\ University of South Carolina \\ Columbia SC 29212 \\ USA
\and Visiting Professor\\ Department of Pure and Applied Mathematics\\ University of Johannesburg\\
P.O. Box 524, Auckland Park, Johannesburg 2006\\South Africa}
\email{czabarka@math.sc.edu}

\address{Josiah Reiswig\\ Department of Mathematics \\ University of South Carolina \\ Columbia SC 29212 \\ USA}
\email{jreiswig@email.sc.edu}

\address{L\'aszl\'o Sz\'ekely\\Department of Mathematics \\ University of South Carolina \\ Columbia SC 29212 \\ USA
\and Visiting Professor\\ Department of Pure and Applied Mathematics\\ University of Johannesburg\\
P.O. Box 524, Auckland Park, Johannesburg 2006\\ South Africa}
\email{szekely@math.sc.edu}

\address{Zhiyu Wang\\ Department of Mathematics \\ University of South Carolina \\ Columbia SC 29212 \\ USA}
\email{zhiyuw@math.sc.edu}



\thanks{This material is based upon work that started at  the workshop ``Biplanar Crossing Numbers and Random Graphs", February 22--25, 2018,
at the University of South Carolina,  with the 
support of the National Science Foundation  contract DMS 1641020 and the American Mathematical Society. This workshop was a follow-up of the Mathematics Research Communities workshop ``Beyond Planarity: Crossing Numbers of Graphs'', Snowbird, Utah, June 11--18, 2017.
The last three authors were also supported in part by the National Science Foundation contract DMS-1600811.}

\maketitle


\begin{abstract}
For positive integers $n$ and $e$, let $\kappa(n,e)$ be the minimum crossing number (the standard planar crossing number) taken over all graphs with $n$ vertices and at least $e$ edges.
Pach, Spencer and T\'oth [\emph{Discrete and Computational Geometry}
{\bf 24} 623--644, (2000)] showed that $\kappa(n,e) n^2/e^3$ tends to a positive constant (called midrange crossing constant) as $n\to \infty$ and $n \ll e \ll n^2$, proving a conjecture of Erd\H{o}s and Guy. In this note, we extend their proof to show that the midrange crossing constant exists for graph classes that satisfy a certain set of graph properties. As a corollary, we show that the the midrange crossing constant exists for the family of bipartite graphs. All these results have their analogues for rectilinear crossing numbers.
\end{abstract}

\section{Introduction}

Given an undirected graph $G = (V,E)$, a \textit{drawing} of $G$ is a representation of $G$ in the plane such that every edge $uv \in E$ is represented by a simple continuous curve between the points corresponding to $u$ and $v$, which does not pass through any point representing a vertex of $G$. For simplicity, we assume that no two curves share infinitely many points, no two curves are tangent to each other, and no three curves pass through the same point. The \textit{crossing number} $\Cr(G)$ is defined as the minimum number of  crossing points in a drawing of $G$. It is well-known \cite{Schaefer} that in any drawing that realizes the crossing number, any pair of edges crosses 
in at most one point, and a pair of crossing edges do have 4 distinct endvertices.
Computing $\Cr(G)$ is an NP-hard problem \cite{GJ}. Motivated by applications to VLSI design, Leighton\cite{Leighton-1}, and independently Ajtai et al. \cite{Ajtai}, gave the following general lower bound on the crossing number of a graph, which is better known as the \textit{crossing lemma}.

\begin{thm}[\cite{Leighton-1, Ajtai}]
For any graph $G$ with $n$ vertices and $e > 4n$ edges, we have 
$$\Cr(G) \geq \frac{1}{64}\frac{e^3}{n^2}.$$
\end{thm}
Further improvements on the constant are made by Pach and T\'oth \cite{PT-cr} and Pach et al \cite{PRTT}. The current best bound is due to Ackerman \cite{Ackerman}, who showed that $\Cr(G) \geq \frac{e^3}{29 n^2}$ when $e > 7n$. Sz\'ekely \cite{Szekely} used the crossing lemma to give a simple proof of the Szemer\'edi--Trotter theorem on the number of point-line incidences \cite{ST}. 

For  a positive integer  $n$ and real number $e \geq 0$, let $\kappa(n,e)$ be the minimum crossing number taken over all graphs with $n$ 
vertices and at least $e$ edges. The crossing lemma implies that for $e > 4n$, $\kappa(n,e) \frac{n^2}{e^3}$ is bounded below by a positive constant.
Pach, Spencer and T\'oth \cite{PST} showed that for $n\ll e\ll n^2$ (i.e. in the midrange),
$$
\lim_{n\rightarrow\infty}\kappa(n,e)\frac{n^2}{e^3}=C>0,
$$
which proves a conjecture of Erd\H{o}s and Guy \cite{EG}, made a decade before the crossing lemma. Here, we use the notation $f(n) \ll g(n)$ to mean that $\lim_{n\to\infty} f(n)/g(n) = 0$. The constant $C$ above is also called the \textit{midrange crossing constant}.

For any positive integer $k\geq 1$, the \textit{k-planar crossing number} $\Cr_k(G)$ of $G$ is defined as the minimum of $\sum_{i=1}^k \Cr(G_i)$, where the minimum is taken over all graphs $G_1, G_2, \ldots, G_k$ such that $\bigcup_{i=1}^k E(G_i) = E(G)$. Pach et al. \cite{PSTT} showed a general bound on the ratio of the $k$-planar crossing number to the (ordinary) crossing number of a graph. They defined   
$$\alpha_k = \textrm{sup } \frac{\Cr_k(G)}{\Cr(G)},$$
where the supremum is taken over all nonplanar graphs $G$. Pach et al.  \cite{PSTT} showed that for every positive integer $k$, 
$$\frac{1}{k^2} \leq \alpha_k \leq \frac{2}{k^2}-\frac{1}{k^3}.$$
For $k = 2$, this gives the same upper bound of $3/8$ as in \cite{CSSV}. Very recently, Asplund et al. \cite{Asplund} closed the gap and showed that 
$$\alpha_k = \frac{1}{k^2} (1+o(1)) \textrm { as $k\to\infty$}.$$
The lower bound that $\alpha_k \geq \frac{1}{k^2}$ in Pach et al. \cite{PSTT} depends on the existence of the midrange crossing number $C > 0$, though not on its value. 

For the family of bipartite graphs, define analogously the constant $\beta_k = \textrm{sup} \frac{\Cr_k(G)}{\Cr(G)}$, where the supremum is taken over all \textit{non-planar bipartite} graphs $G$. Asplund et al. \cite{Asplund} showed that for all positive integers $k$, 
$$\beta_k = \frac{1}{k^2}.$$

As before, the lower bound $\beta \geq \frac{1}{k^2}$, depends on the existence of the midrange crossing constant $C_B > 0$ for the family of bipartite graphs. This motivates us to extend the proof of Pach, Spencer and T\'oth \cite{PST} to show the existence of the midrange crossing constant $C_{\class}$ for certain graph classes $\class$, which may or not be equal to the midrange crossing constant $C$ for all graphs. The current best known bounds for $C$ are 
$0.034 \leq C \leq 0.09$; see \cite{PT-cr,PRTT, Ackerman}, while Angelini, Bekos, Kaufmann, Pfister and Ueckerdt \cite{ABKPU} showed that the midrange crossing constant for the class of bipartite graphs is at least $16/289 > 0.055$.

For a class of graphs $\mathcal{B}$,  define $\kappa_{\mathcal{B}}(n,e)$ to be the minimum  crossing number of a graph in $\mathcal{B}$ with $n$ vertices and at least $e$ edges. The following natural questions arise:
\begin{quest}\label{q:exist}
For a given class $\mathcal{B}$, does there exist a constant $C_\mathcal{B}$ such that $\lim\limits_{n\rightarrow\infty}\kappa_{\mathcal{B}}(n,e)\frac{n^2}{e^3}=C_{\mathcal{B}}$ in the midrange?
\end{quest}
\begin{quest}\label{q:diff}
Are there two classes of graphs $\mathcal{B}$ and $\mathcal{D}$ such that $C_{\mathcal{B}}$ and $C_{\mathcal{D}}$ exist with $C_{\mathcal{B}}\neq C_{\mathcal{D}}$?
\end{quest}

To this end, we define
 \begin{defn}\label{defn:PST}
A family of graphs $\class$ is a PST-class (abbreviating Pach, Spencer and T\'oth) if it satisfies the following properties: 
\begin{enumerate}[{\rm(1)}]
\item $\class$ contains a graph with at least one edge;
\item $\class$ is closed under taking subgraphs;
\item $\class$ is closed under taking the union of disjoint copies of graphs in $\class$;

\item $\class$ is closed under vertex cloning, i.e. $\class$ is closed under replacing a vertex $v$ with vertices $v_1,\ldots,v_m$ such that $N(v_i)=N(v)$ for all $1\le i\le m$. 
\end{enumerate}
 \end{defn}

Note that properties (2) and (4) imply that a PST-class  $\class$ is also closed under vertex-splitting i.e. $\class$ is closed under replacing a vertex $v$ with vertices $v_1,v_2,\ldots,v_m$ such that $N(v)=\bigcup_{i=1}^m N(v_i)$ and $N(v_i)\cap N(v_j)=\emptyset$ for $1\leq i <j \leq m$;

PST classes form a lattice with respect to the subclass relation, with the set of bipartite graphs being the minimum and the set of all graphs being a maximum element of the class:
\begin{thm}\label{thm:main2}
Any PST-class contains the family of bipartite graphs as a subclass, and the intersection of two PST-classes is also a PST class.
Moreover, the following are PST classes:
\begin{enumerate}[{\rm (a)}]
\item $\ell$-colorable graphs;
\item $K_t$-free graphs ($t\ge 3$);
\item graphs without odd cycles shorter than $g$. 
\end{enumerate}
\end{thm}

Careful examination of the proof presented in \cite{PST} yields that PST-classes have midrange crossing constants:

\begin{thm}\label{thm:main}
If $\class$ is a PST-class, then $C_{\class}$ exists, i.e. there is a constant $C_{\class}>0$ such that in the midrange
$$\lim\limits_{n\rightarrow\infty}\kappa_{\mathcal{B}}(n,e)\frac{n^2}{e^3}=C_{\mathcal{B}}.$$
\end{thm}

For Question~\ref{q:diff}, the answer remains elusive. By Theorem~\ref{thm:main2}, if we restrict our attention to PST-classes, an affirmative answer implies that the midrange crossing constant for the class of bipartite graphs is bigger than
 the midrange crossing constant for all graphs, which we tend to believe. 

Pach et al. \cite{PSTT} pointed out that the arguments
in \cite{PST} can be repeated for rectilinear drawings of graphs and rectilinear crossing numbers; therefore a midrange rectilinear crossing constant exists, which is not necessarily the same as the midrange crossing constant. Extending this argument, Theorems~\ref{thm:main} and \ref{thm:main2} have their rectilinear versions, with possibly different constants.

The proof of Theorem \ref{thm:main} closely follows the original proof for the existence of the midrange crossing constant in \cite{PST}, with  emphasis added on the required properties of the graph class.

\section{Proof of Theorem \ref{thm:main2}}

If $\class$ is a PST-class, properties (1) and (2) imply that $K_2\in\class$. Property (4) then gives that $K_{n,m}\in\class$ for all $n,m$, and property (2) then implies that
all bipartite graphs are elements of $\class$. Let $\class_1,\class_2$ be PST-classes, then their intersection must contain all bipartite graphs, and therefore a graph, which is not
edgeless. If $G^*$ is a graph obtained from a graph in $G\in\class_1\cap\class_2$ by taking subgraphs, or vertex cloning, then, as
$\class_1$ and $\class_2$ are closed under these operations, $G^*\in\class_1\cap\class_2$. As $\class_1,\class_2$ are closed under taking disjoint unions, so is
$\class_1\cap\class_2$; so $\class_1\cap\class_2$ is a PST-class.

$K_2$ is in the classes of $\ell$-colorable graphs, $K_t$-free graphs for $t\ge 3$, and graphs without odd cycles shorter then $g$, so these three classes satisfy property (1). It is obvious that taking a subgraph does not increase the chromatic number, the clique number or the length of the shortest odd cycle. The chromatic number, clique number and length of the shortest odd cycle are the minimum of these quantities respectively over the components of a graph, so these classes
are closed under property (3). 

Let $G^*$ be obtained from $G$ by  cloning $v$ into $v_1,\ldots,v_m$. Given a good $\ell$-coloring of $G$, we can obtain a good $\ell$-color of $G^*$ by assigning the color of $v$ to all $v_i$; any complete subgraph of $G^*$ can contain at most one $v_i$ and thus  correspond to a complete subgraph of the same order in $G$.
Any cycle in $G^*$ that contains at most one of the $v_i$'s corresponds to a cycle  of the same length in $G$.
Any cycle in $G^*$ that contains more than one $v_i$ corresponds to a closed walk in $G$ that visits all vertices (of the walk) but $v$ at most once, so it corresponds to the union of
cycles in $G$; this means that for  any odd cycle in $G^*$ we can find an odd cycle in $G$ that is not longer. This shows that these classes are all PST-classes.

\section{Proof of Theorem \ref{thm:main}}

\noindent 
Throughout this section, $\class$ denotes a PST-class.


We first show a lemma that bound the minimum number of crossings of a graph in $\class$ with $n$ vertices and linearly many edges.

\begin{lem}\label{lemma: cross-bound}
For any $a \geq 4$, and $n\geq 2a+1$, 
$$\frac{a^3}{100} \leq \frac{\kappa_\class(n,an)}{n} \leq 8a^3.$$
\end{lem}

\begin{proof}
It is known (see \cite{PT-cr}) that for any $a \geq 4$, $\frac{{a^3}n}{100} \leq \kappa(n,an) \leq a^3n$. Since $\kappa(n,a n) \leq \kappa_{\class}(n,a n) $ for any $a > 0$, it follows that $\frac{a^3 n}{100} \leq \kappa_{\class}(n,a n)$.

Since  $\kappa(n,2a n)\le 8a^3n$, we know that there exists a graph $G$ on $n$ vertices and $\ceil{2a n}$ edges such that $\Cr(G) \leq 8a^3 n$. As it is well known, every graph has a bipartite subgraph with at least half as many edges. Hence there exists a bipartite subgraph $H$ of $G$, such that $H$ has at least ${an}$ edges and $\Cr(H)\leq \Cr(G)\leq 8a^3n$. Since $\class$ is closed under taking subgraphs, $H\in\class$. Thus, for any $a \geq 4$, we have that 
$$\frac{a^3}{100} \leq \frac{\kappa_{\class}(n,a n)}{n} \leq \frac{\Cr(H)}{n} \leq 8a^3.$$
\end{proof}

\begin{lem}
We have the following:
\begin{enumerate}[{\rm (a)}]
\item For any $a>0$, the limit 
$$
\gamma_\class[a]=\lim_{n\ra\infty}\dfrac{\kappa_\class(n,na)}{n}
$$
exists and is finite.
\item $\gamma_\class[a]$ is a convex function.
\item For any $a\geq 4$, $1>\delta>0$,
$$
\gamma_\class[a]-\gamma_\class[a(1-\delta)]\leq \gamma_\class[a(1+\delta)]-\gamma_\class[a]\leq 10^4\delta\gamma_\class[a].
$$
\item $\gamma_\class [a]$ is continuous.
\end{enumerate}
\end{lem}

\begin{proof}
Let $G_1,G_2\in\class$, and let $G_3$ be their disjoint union. As $\Cr(G_3)=\Cr(G_1)+\Cr(G_2)$ and
$\class$ is closed under taking disjoint union,
$$
\kappa_\class(n_1+n_2,e_1 + e_2)\leq \kappa_\class(n_1,e_1)+\kappa_\class(n_2,e_2).
$$
This implies that $f_a(n):=\kappa_\class(n,na)$ is subadditive. 
Hence, by Fekete's Lemma on subadditive sequences (see \cite{polya} Ex. 98),
$$
\gamma_\class[a]=\lim_{n\ra\infty}\frac{\kappa_\class(n,na)}{n}
$$
exists and is finite for every $a>0$. This concludes the proof of part (a).

Let $0<\alpha<1$. Suppose $a,b>0$ with $n$ such that $\alpha n$ is also an integer. 
Then,
\begin{align*}
\kappa_\class(n,(\alpha a+(1-\alpha)b)n)
&=\kappa_\class(\alpha n + (1-\alpha)n,(\alpha a+(1-\alpha)b)n)\\
&\leq\kappa_\class(\alpha n,\alpha a n)+ \kappa_\class((1-\alpha) n,(1-\alpha)b n).
\end{align*}  
This implies that 
\begin{align*}
\gamma_\class[\alpha a+(1-\alpha)b]
&=\lim\limits_{n\rightarrow\infty}\dfrac{\kappa_\class(n,(\alpha a+(1-\alpha)b)n)}{n}\\
&\leq \lim\limits_{n\rightarrow\infty,\alpha n\in \mathbb{Z}}\dfrac{\kappa_\class(\alpha n,\alpha a n)}{n}+ \lim\limits_{n\rightarrow\infty,\alpha n\in \mathbb{Z}}\dfrac{\kappa_\class((1-\alpha) n,(1-\alpha)b n)}{n}.
\end{align*}
Now, for any rational number $\alpha$, let $m=\alpha n$ and $q=(1-\alpha)n=n-m$. Then,
\begin{align*}
&\lim\limits_{n\rightarrow\infty,\alpha n\in \mathbb{Z}}\dfrac{\kappa_\class(\alpha n,\alpha a n)}{n}+ \lim\limits_{n\rightarrow\infty, \alpha n\in \mathbb{Z}}\dfrac{\kappa_\class((1-\alpha) n,(1-\alpha)b n)}{n}\\
&=\lim\limits_{m\rightarrow\infty}\dfrac{\kappa_\class(m,a m)}{(m/\alpha)}+\lim\limits_{q\rightarrow\infty}\dfrac{\kappa_\class(q,b q)}{q/(1-\alpha)}\\
&=\alpha\lim\limits_{m\rightarrow\infty}\dfrac{\kappa_\class(m,a m)}{m}+(1-\alpha)\lim\limits_{q\rightarrow\infty}\dfrac{\kappa_\class(q,b q)}{q}
=\alpha\gamma_\class[a]+(1-\alpha)\gamma_\class[b].
\end{align*}
Hence, for any rational $\alpha$
$$
\gamma_\class[\alpha a+(1-\alpha)b]\leq \alpha\gamma_\class[a]+(1-\alpha)\gamma_\class[b]. 
$$

We now relax the restriction that $\alpha$ is rational. Suppose that $a\leq b$ and let $\{\alpha_n\}_{n=0}^\infty$ be a monotone increasing sequence of rational numbers converging to $\alpha$. For any $n\geq 0$, 
\begin{align*}
\gamma_\class[\alpha a + (1-\alpha)b]&\leq \gamma_\class[\alpha_n a + (1-\alpha_n)b]\\
&\leq \alpha_n\gamma_\class[a]+(1-\alpha_n)\gamma_\class[b],
\end{align*}
which converges to $\alpha\gamma_\class[a]+(1-\alpha)\gamma_\class[b]$.
So, 
$$
\gamma_\class[\alpha a+(1-\alpha)b]\leq \alpha\gamma_\class[a]+(1-\alpha)\gamma_\class[b]
$$
for all real numbers $0<\alpha<1$. That is, $\gamma_\class$ is convex, so part (b) is true, which implies
$$
\gamma_\class[a]-\gamma_\class[a(1-\delta)]\leq \gamma_\class[a(1+\delta)]-\gamma_\class[a],
$$ 
which is the first inequality in part (c). For the second inequality, by Lemma \ref{lemma: cross-bound}, we obtain that 
$$\frac{a^3}{100} \leq \gamma_{\class}(a) \leq 8a^3.$$
Then, for $a\geq 4$ and $0<\delta<1$, 
\begin{align*}
\gamma_\class[(1+\delta)a]
&=\gamma_\class[(1-\delta)a+2\delta a]\\
&\leq (1-\delta)\gamma_\class[a]+\delta \gamma_\class[2a]\\
&\leq \gamma_\class[a]+\delta \gamma_\class[2a].
\end{align*}
Hence,
\begin{align*}
\gamma_\class[(1+\delta)a]-\gamma_\class[a]
\leq \delta \gamma_\class[2a]
\leq \delta \cdot  64 a^3
\leq \delta \cdot 6400 \gamma_\class [a]
< 10^4 \delta \cdot \gamma_\class [a],
\end{align*}
which finishes the proof of part (c) and implies part (d).
\end{proof}

We now set 
$$
D_\class:=\limsup_{a\rightarrow\infty}\frac{\gamma_\class[a]}{a^3}.
$$
Note that by Lemma~\ref{lemma: cross-bound} we have $0.01\le D_\class\le 8$. We will follow the proof of Theorem \ref{thm:main}
presented in \cite{PST} by showing two lemmas  that imply that 
$D_\class$ is the midrange crossing constant $C_\class$. 

\begin{lem}
For any $0 < \epsilon<0.1$, there exists $N=N(\epsilon)$ such that $\kappa_\class(n,e)>(1-\epsilon)\frac{e^3}{n^2} D_\class$, whenever $\min\{n,e/n,n^2/e\}>N$.
\end{lem}

\begin{proof}
Let $A>\frac{2\cdot 10^9}{\epsilon^3}$ be a rational number satisfying 
$$
\frac{\gamma_\class[A]}{A^3}>D_\class\left(1-\frac{\epsilon}{10}\right).
$$
Such a number exists by the definition of $D_\class$. Let $N=N(\epsilon)\geq A$ such that if $n>N$, $e=nA'$, and $|A-A'|\leq A\epsilon$, then
\begin{equation}
\kappa_\class(n,e)>\gamma_\class[A']\left(1-\frac{\epsilon}{10}\right)n. \label{eqn}
\end{equation}
Such an $N$ certainly exists, as this is equivalent to
$$
\frac{\kappa_\class(n,e)}{n}>\lim_{k\ra\infty}\frac{\kappa_\class(k,A'k)}{k}\left(1-\frac{\epsilon}{10}\right).
$$ 

Let $n$ and $e$ be fixed integers so that $\min\{n,e/n,n^2/e\}>N$ and let $G=(V,E)$ be a graph with $|V|=n$ vertices and $|E|=e$ edges, which can be drawn in the plane with $\kappa_\class(n,e)$ crossings. Set $p=An/e$ and let $U$ be a randomly chosen subset of $V$ with $P[u\in U]=p$, independently for each $v\in V$. Let $\nu=|U|$ and let $\eta$ and $\xi$ be the number of edges and crossings (in the drawing) of the graph induced by $U$. We have that $\nu$ has expected value $np$ and variance $p(1-p)n\leq pn$. By Chebyshev's inequality, we have that
\begin{align*}
\op{Pr}\left[|\nu-pn|>\frac{\epsilon}{10^4}pn\right]
&\leq \frac{10^8pn(1-p)}{(pn)^2\epsilon^2}
\leq \frac{10^8}{pn\epsilon^2}
\leq \frac{\epsilon^3 10^8 e}{2\cdot 10^9n^2\epsilon^2}
< \frac{\epsilon}{10}.
\end{align*}
We note that $\eta=\sum_{u,v\in G} I_{uv}$ where $I_{uv}$ is the indicator function for the event $u,v\in U$. Then, $E[\eta]=ep^2$. Since $I_{uv}$ is an indicator function, we have that
\begin{align*}
E[I_{uv}]&=p^2\\
\op{Var}[I_{uv}]&=p^2(1-p^2)\\
\op{Cov}[I_{uv},I_{wx}]&=\op{Pr}[(u,v\in U)\cap (w,x\in U)]-\op{Pr}(u,v\in U)\cdot \op{Pr}(w,x\in U).
\end{align*}
It is immediate that $Var(I_{uv})\leq E(I_{uv})$ and that $\op{Cov}[I_{uv},I_{wx}]=0$ if $\{u,v\}\cap \{w,x\}=\emptyset$. Hence,
$$
\op{Var}[\eta]=\sum_{uv\in E}\op{Var}[I_{uv}]+\sum_{uv,uw\in E}\op{Cov}[I_{uv},I_{uw}],
$$
and$$
\sum_{uv\in E}\op{Var}[I_{uv}]\leq \sum_{uv\in E}E[I_{uv}]=E[\eta]=ep^2
$$
Using the bound $\op{Cov}[I_{uv},I_{uw}]\leq E[I_{uv}I_{uw}]=p^3$, we see that
$$
\op{Var}[\eta]\leq p^2e+p^3\sum_{v\in V}\binom{d(v)}{2}.
$$
Since $d(v)<n$ and $\sum_{v\in V}d(v)=2e$, we have that 
$$
\sum_{v\in V}\binom{d(v)}{2}\leq \frac{1}{2}\sum_{v\in V} d^2(v)<\frac{1}{2}\sum_{v\in V}d(v)n= en. 
$$
Hence, using that $pn=\frac{An^2}{e}>AN>1$, 
$$
\op{Var}[\eta]\leq p^2e+p^3en\leq 2p^3en.
$$
Applying the Chebyshev Inequality, and using that $pe=An>\frac{2\cdot 10^9}{\epsilon^3}$, we see that
\begin{align*}
\op{Pr}\left[|\eta -p^2e|>\frac{\epsilon}{10^4}p^2e\right]
&\leq \frac{2p^3en}{\frac{\epsilon^2}{10^8}p^4e^2}
= \frac{2n}{\frac{\epsilon^2}{10^8}pe}
< \frac{2}{\frac{\epsilon^2}{10^8}\frac{2\cdot 10^9}{\epsilon^3}}
= \frac{\epsilon}{10}.
\end{align*}
This implies that with probability at least $1-\frac{\epsilon}{5}$, 
$$
pn\left(1-\frac{\epsilon}{10^4}\right)<\nu< pn\left(1+\frac{\epsilon}{10^4}\right),
\mbox{ and }
p^2e\left(1-\frac{\epsilon}{10^4}\right)<\eta<p^2e\left(1+\frac{\epsilon}{10^4}\right).
$$
Hence, with probability at least $1-\frac{\epsilon}{5}$, 
\begin{align*}
\frac{p^2e}{pn}\cdot\frac{1-\frac{\epsilon}{10^4}}{1+\frac{\epsilon}{10^4}}<&\frac{\eta}{\nu}<\frac{p^2e}{pn}\cdot\frac{1+\frac{\epsilon}{10^4}}{{1-\frac{\epsilon}{10^4}}},
\end{align*}
which implies that 
$$
A\left(1-\frac{3\epsilon}{10^4}\right)<\frac{\eta}{\nu}<A\left(1+\frac{3\epsilon}{10^4}\right).
$$
Now we set $A'=\frac{\eta}{\nu}$. The subgraph induced by $U$ has $\nu$ vertices and $A'\nu$ vertices. So, with probability at least $1-\frac{\epsilon}{5}$, equation \eqref{eqn} implies that the number of crossings in this induced subgraph is at least 
$$
\nu\gamma_\class[A']\left(1-\frac{\epsilon}{10}\right)\geq pn\left(1-\frac{\epsilon}{10}\right)\gamma_\class[A']\left(1-\frac{\epsilon}{10}\right).
$$ 
Then, the expected number of crossings in the subgraph induced by $U$ in $G$ is at least
\begin{align*}
E[\xi]
&\geq \left(1-\frac{\epsilon}{5}\right)pn\left(1-\frac{\epsilon}{10}\right)\gamma_\class[A']\left(1-\frac{\epsilon}{10}\right)\\
&\geq \left(1-\frac{\epsilon}{5}\right)pn\left(1-\frac{\epsilon}{10}\right)\gamma_\class[A]\left(1-\frac{3\epsilon}{10}\right)\left(1-\frac{\epsilon}{10}\right)\\
&> \left(1-\frac{\epsilon}{5}\right)pn\left(1-\frac{\epsilon}{10}\right)D_\class A^3\left(1-\frac{3\epsilon}{10}\right)\left(1-\frac{\epsilon}{10}\right)\left(1-\frac{\epsilon}{10}\right)\\
&\geq (1-\epsilon)D_\class A^3pn.
\end{align*}
However, since each crossing lies in $G[U]$ with probability $p^4$, we know that 
$$
E[\xi]=p^4\kappa_\class(n,e).
$$
Hence, 
$$
\kappa_\class(n,e)\geq (1-\epsilon)\frac{pnD_\class A^3}{p^4}=\frac{e^3}{n^2}(1-\epsilon)D_\class.
$$
\end{proof}

\begin{lem}
For any $0 < \epsilon<0.1$, there exists $M=M(\epsilon)$ such that $\kappa_{\class}(n,e)<(1+\epsilon)\frac{e^3}{n^2} D_\class $, whenever $\min\{n,e/n,n^2/e\}>M$. 
\end{lem}

\begin{proof}
Let $A>10^8/\epsilon^2$ be a rational number satisfying
$$
D_\class\left(1-\frac{\epsilon}{10}\right)<\frac{\gamma_\class[A]}{A^3}<D_\class\left(1+\frac{\epsilon}{10}\right)
$$
such that $A^{3/2}$ is an integer. Let $M_1=M_1(\epsilon) \geq A$ such that, if $n>M_1$ and $e=nA$, then
$$
D_\class A^3n\left(1-\frac{\epsilon}{5}\right)<\kappa_\class(n,e)<D_\class A^3n\left(1+\frac{\epsilon}{5}\right).
$$
Let $G_1$ be a graph in $\class$ with $n_1>M_1$ vertices $e_1=An_1$ edges, and suppose that $G_1$ is drawn in the plane with $\kappa_\class(n_1,e_1)$ crossings, where 
$$
D_\class A^{3}n_1\left(1-\frac{\epsilon}{5}\right)<\kappa_\class(n_1,e_1)<D_\class A^3n_1\left(1+\frac{\epsilon}{5}\right).
$$
For each vertex $v\in G_1$ 
do the following. Let $d(v)=r_vA^{3/2}+s_v$ where $r_v,s_v$ are integers and $0\leq s_v< A^{3/2}$. We split $v$ into $r_v+1$ vertices, one with degree $s_v$ and $r_v$ with degree $A^{3/2}$. (Note that this implies that we do nothing for vertices with $r_v=0$, i.e. when the degree of the vertex is smaller than $A^{3/2}$.)
Drawing these vertices very close to each other, we may do this without creating any additional crossings. This creates a drawing of a new graph $G_2$ that has $n_2$ vertices, $e_1$ edges and maximum degree at most $A^{3/2}$, and the crossing number of this drawing is $\kappa_\class(n_1,e_1)$. Since $\class$ is closed under vertex splitting, $G_2\in \class$. We have that 
$$
2A n_1=2e_1=\sum_{v\in G_1}d(v)=\sum_{v\in V}(r_vA^{3/2}+s_v),
$$
which implies that $\sum_{v\in V}r_v\leq \frac{2n_1}{\sqrt{A}}.$
Hence, 
$$
n_1\leq n_2\leq n_1+\frac{2n_1}{\sqrt{A}}\leq n_1\left(1+\frac{\epsilon}{10}\right). 
$$
Now, fix integers $n$ and $e$ such that $\min\{n,e/n,n^2/e\}>M(\epsilon)=10\frac{M_1}{\epsilon}$. Let 
$$
L=\frac{e/n}{e_1/n_2}
\quad
\mbox{and} 
\quad
K=\frac{n^2/e}{n_2^2/e_1},
$$
so that $n=KLn_2$ and $e=KL^2e_1$. Let
$$
\tilde{L}=\floor{L\left(1+\frac{\epsilon}{10}\right)}\quad \mbox{and}\quad \tilde{K}=\floor{K\left(1-\frac{\epsilon}{10}\right)},
$$
and let 
$$
\tilde{n}=\tilde{K}\tilde{L}n_2\quad \mbox{and}\quad \tilde{e}=\tilde{K}\tilde{L}^2e_1.
$$
Then,
$$
\tilde{n}<n \mbox{ and } e<\tilde{e}. 
$$
which implies that $\kappa_\class(n,e)<\kappa_\class(\tilde{n},\tilde{e})$. Substitute each vertex of $G_2$ with $\tilde{L}$ very close vertices, and substitute each edge of $G_2$ with the corresponding $\tilde{L}^2$ edges, each very close to the original, obtaining a drawing of a new graph $G_3$ with
$n_2\tilde{L}$ verices and $e_1 \tilde{L}^2$ edges. As $G_3$ is obtained from $G_2$ by cloning each vertex, $G_3\in\class$. 
Then make $\tilde{K}$ copies of this drawing, each separated from the others. We then have a graph $\tilde{G}$
on $\tilde{n}$ vertices and $\tilde{e}$ edges drawn in the plane. As $\tilde{G}$ is a disjoint union of $\tilde{K}$ copies of $G_3$, $\tilde{G}\in\class$. We will estimate the number of crossings $X$ in this drawing. 

A crossing in the original drawing of $G_2$ corresponds to $\tilde{K}\tilde{L}^4$ crossings in the present drawing of $\tilde{G}$. For any two edges of $G_2$ with a common endpoint, the edges arising from them have at most $\tilde{K}\tilde{L}^4$ crossings with each other. So,
$$
X\leq \tilde{K}\tilde{L}^4\left(\kappa_\class(n_1,e_1)+\sum_{v\in V(G_2)}\binom{d(v)}{2}\right).
$$
However, 
$$
\sum_{v\in V(G_2)}d(v)=2e_1=2An_1  
$$
and $d(v) \leq A^{3/2}$,
so 
$$
\sum_{v\in V(G_2)}\binom{d(v)}{2}\leq  A^{5/2}n_1<\frac{\kappa_\class(n_1,e_1)}{A^{1/2}D_\class\left(1-\frac{\epsilon}{5}\right)}\le\frac{\epsilon}{10}\kappa_\class(n_1,e_1).
$$
Therefore, 
\begin{align*}
\kappa_\class(n,e)&<\kappa_\class(\tilde{n},\tilde{e})<\tilde{K}\tilde{L}^4\kappa_\class(n_1,e_1)\left(1+\frac{\epsilon}{10}\right)
<\tilde{K}\tilde{L}^4D_\class A^3n_1\left(1+\frac{\epsilon}{5}\right)\left(1+\frac{\epsilon}{10}\right)\\
&=\tilde{K}\tilde{L}^4D_\class\frac{e_1^3}{n_1^2}\left(1+\frac{\epsilon}{5}\right)\left(1+\frac{\epsilon}{10}\right)
<KL^4D_\class\frac{e_1^3}{n_2^2}\left(1+\frac{\epsilon}{10}\right)^6\left(1+\frac{\epsilon}{5}\right)\left(1+\frac{\epsilon}{10}\right)\\
&=\frac{e^3}{n^2}\left(1+\frac{\epsilon}{10}\right)^6\left(1+\frac{\epsilon}{5}\right)\left(1+\frac{\epsilon}{10}\right)D_\class
<(1+\epsilon)\frac{e^3}{n^2}D_\class.
\end{align*}
\end{proof}

\end{document}